\documentclass[11pt]{article} 

\usepackage[utf8]{inputenc} 

\usepackage{geometry} 
\usepackage{graphicx} 

\usepackage{booktabs} 
\usepackage{array} 
\usepackage{paralist} 
\usepackage{verbatim} 
\usepackage{subfig} 

\usepackage{fancyhdr} 
\usepackage[nottoc,notlof,notlot]{tocbibind} 
\usepackage[titles,subfigure]{tocloft}

\usepackage{amsmath,amsfonts,amsthm,mathrsfs,amssymb}
\usepackage[usenames]{color}

\newtheorem{thm}{Theorem}[section]

\newtheorem{lem}{Lemma}[section]

\theoremstyle{definition}
\newtheorem{defn}{Definition}[section]
\theoremstyle{remark}
\newtheorem{rem}{Remark}[section]
\numberwithin{equation}{section}

  \def\nb{\nonumber}

\def\Om{\Omega}   
\newcommand{\q}{\quad}

\def\bb{\begin{equation}} \def\ee{\end{equation}}
\def\beqn{\begin{eqnarray}}  \def\eqn{\end{eqnarray}}
\def\beqnx{\begin{eqnarray*}} \def\eqnx{\end{eqnarray*}}

\title{ Nearly Cloaking the Full Maxwell Equations: \\
Cloaking Active Contents with General Conducting Layers}
\author{Gang Bao\thanks{Department of Mathematics, Zhejiang University, Hangzhou 310027,
China; Department of Mathematics, Michigan
State University, East Lansing, MI 48824. Email: {\tt bao@math.msu.edu} } \quad Hongyu Liu\thanks{Department of Mathematics and Statistics, University of North Carolina, Charlotte, NC 28223, USA.   Email:  {\tt hongyu.liuip@gmail.com}} \quad Jun Zou\thanks{Department of Mathematics, The Chinese University of Hong Kong, Shatin, Hong Kong. Email: {\tt zou@math.cuhk.edu.hk}}  }

\begin{document}


%
%
%
%
%
%

\date{}

\maketitle

\begin{abstract}

The regularized near-cloak via the transformation optics approach in the time-harmonic electromagnetic scattering is considered. This work extends the existing studies mainly in two aspects. First, it presents a near-cloak construction by incorporating a much more general conducting layer between the cloaked and cloaking regions. This might be of significant practical importance when production fluctuations occur. Second, it allows the cloaked contents 
to be both passive and active with an applied current inside. 
In assessing the near-cloaking performance, 
comprehensive and sharp estimates are derived for the scattering amplitude in terms of the asymptotic regularization parameter and the material tensors of the conducting layer. The scattering estimates are independent of the passive/active contents being cloaked, which implies that one could nearly cloak arbitrary contents by using the proposed near-cloak construction.

\medskip


\end{abstract}

\section{Introduction}

This work is concerned with the invisibility cloaking for electromagnetic (EM) waves via the approach of transformation optics \cite{GLU,GLU2,Leo,PenSchSmi}. This is a rapidly growing research area 
with many potential applications, and 
we refer to  \cite{CC,GKLU4,GKLU5,Nor,U2,YYQ} and the references therein
for the recent progress on both the theory and experiments.

We consider two bounded simply connected smooth domains $D$ and $\Omega$ 
in $\mathbb{R}^3$, with $D\Subset\Omega$, 
and three real symmetric matrix-valued functions  
$\tilde\varepsilon=(\tilde\varepsilon^{ij})_{i,j=1}^3$, $\tilde\mu=(\tilde\mu^{ij})_{i,j=1}^3$ and $\tilde\sigma=(\tilde\sigma^{ij})_{i,j=1}^3$ in $\Omega$, satisfying 
\begin{equation}\label{eq:uniform elliptic}
c|\xi|^2\leq \sum_{i,j=1}^3 \tilde\varepsilon^{ij}(x)\xi_i\xi_j\leq C |\xi|^2,\quad c|\xi|^2\leq \sum_{i,j=1}^3 \tilde\mu^{ij}(x)\xi_i\xi_j\leq C |\xi|^2
\end{equation}
and
\begin{equation}\label{eq:uniform elliptic 2}
0\leq \sum_{i,j=1}^3 \tilde\sigma^{ij}(x)\xi_i\xi_j\leq C |\xi|^2,
\end{equation}
for all $x\in\Omega$ and $\xi=(\xi_i)_{i=1}^3\in\mathbb{R}^3$. Here the constants 
$c$ and $C$, or $c_l$ and $C_l$  for $l=0,1,2$ in the rest of the work, 
are used for generic positive constants whose meanings should be clear from the contexts. Physically, functions $\tilde\varepsilon$, $\tilde\mu$ and $\tilde\sigma$ stand respectively for the electric permittivity, magnetic permeability and conductivity tensors of a {\it regular} EM medium occupying $\Omega$. In this work, we shall often refer to \eqref{eq:uniform elliptic} and \eqref{eq:uniform elliptic 2} as the {\it regular conditions} for an EM medium, and write $(\Omega;\tilde\varepsilon,\tilde\mu,\tilde\sigma)$ for an EM medium residing in $\Omega$. We always assume that the EM medium inclusion $(\Omega;\tilde\varepsilon,\tilde\mu,\tilde\sigma)$ is located in a uniformly homogeneous space where the EM parameters are given by $\varepsilon_0, \mu_0$ and $\sigma_0$. It is assumed that 
$\sigma_0^{ij}=0$ and $\varepsilon_0^{ij}=\mu_0^{ij}=\delta^{ij}$ for the ease of our exposition, 
where $\delta^{ij}$ is the Kronecker delta function. We shall be concerned with an EM medium distribution in the whole space $\mathbb{R}^3$ as follows:
\begin{equation}\label{eq:EM medium}
\mathbb{R}^3; \tilde\varepsilon, \tilde\mu, \tilde\sigma=\begin{cases}
\varepsilon_0, \mu_0, \sigma_0\qquad & \mbox{in\ \ $\mathbb{R}^3\backslash\overline{\Omega}$},\\
\varepsilon_c, \mu_c, \sigma_c\qquad & \mbox{in\ \ $\Omega\backslash\overline{D}$},\\
\tilde\varepsilon_a, \tilde\mu_a, \tilde\sigma_a\qquad & \mbox{in\ \ $D$},
\end{cases}
\end{equation}
where the mediums in $D$ and $\Omega\backslash\overline{D}$ will be specified appropriately 
in the sequel wherever it is necessary. 

Next, we consider the EM scattering corresponding to an EM medium described in \eqref{eq:EM medium}. To that end, we first introduce the governing equations. 
Let $\omega\in\mathbb{R}_+$ be the wave frequency, and $E^i, H^i\in\mathbb{C}^3$ 
be the incident EM fields that are 
(real analytic) entire solutions to the time-harmonic Maxwell equations
\begin{equation}\label{eq:Maxwell incident}
\nabla\wedge E^i-i\omega\mu_0 H^i=0\,, \q 
\nabla\wedge H^i+i\omega \varepsilon_0 E^i=0
\qquad\mbox{in\ \ $\mathbb{R}^3$.}
\end{equation}
Then the EM wave propagation in the whole space $\mathbb{R}^3$ with 
an EM medium inclusion $(\Omega;\tilde\varepsilon,\tilde\mu,\tilde\sigma)$ 
as described in \eqref{eq:EM medium} 
is governed by the following Maxwell system
\begin{equation}\label{eq:Maxwell whole}
\begin{cases}
\displaystyle{\nabla\wedge\widetilde E-i\omega\tilde\mu\widetilde H=0}\qquad\qquad &\mbox{in\ \ $\mathbb{R}^3$,}\\
\displaystyle{\nabla\wedge\widetilde H+i\omega\left(\tilde\varepsilon+i\frac{\tilde\sigma}{\omega}\right)\widetilde E=\widetilde{J}}\quad &\mbox{in\ \ $\mathbb{R}^3$},\\
\displaystyle{\widetilde E^-=\widetilde E|_{\Omega},\quad \widetilde{E}^+=(\widetilde{E}-E^i)|_{\mathbb{R}^3\backslash\overline{\Omega}}}\\
\displaystyle{\widetilde H^-=\widetilde H|_{\Omega},\quad \widetilde{H}^+=(\widetilde{H}-H^i)|_{\mathbb{R}^3\backslash\overline{\Omega}}}\\
\displaystyle{\lim_{|x|\rightarrow+\infty}|x|\left| (\nabla\wedge\widetilde E^+)(x)\wedge\frac{x}{|x|}-i\omega \widetilde E^+(x) \right|=0}
\end{cases}
 \end{equation}
 where $\widetilde{J}\in\mathbb{C}^3$ denotes an electric current density, and 
$supp(\widetilde J)\subset\Omega$. In \eqref{eq:Maxwell whole}, $\widetilde E$ and $\widetilde H$ are respectively the electric and magnetic fields, and $\widetilde E^+$ and $\widetilde H^+$  are known as the scattered fields (cf. \cite{ColKre,Ned}). The last relation 
in \eqref{eq:Maxwell whole} is called the Silver-M\"uller radiation condition, which characterizes the radiating nature of the scattered wave fields $\widetilde E^+$ and $\widetilde H^+$. For a regular EM medium $(\Omega; \tilde\varepsilon, \tilde\mu, \tilde\sigma)$ and an active electric current 
$\widetilde J\in L^2(\Omega)^{3}$, there exists a unique pair of solutions $\widetilde E, \widetilde H\in 
H_{loc}(\nabla\wedge; \mathbb{R}^3)$ (see \cite{Lei,Ned}), and 
$\widetilde E^+$ admits the asymptotic expression as $|x|\rightarrow \infty$ (cf. \cite{ColKre}):
\begin{equation}\label{eq:farfield}
\widetilde E^+(x)=\frac{e^{i\omega |x|}}{|x|} A_\infty\left(\frac{x}{|x|}; E^i\right)+\mathcal{O}\left(\frac{1}{|x|^2}\right)
\end{equation}
where $A_\infty(\hat{x}; E^i)$ with $\hat{x}:=x/|x|\in\mathbb{S}^2$ is known as the {\it scattering amplitude}. 
In the above and sequel, we shall often use the spaces 
\[
H_{loc}(\nabla\wedge; X)=\{ U|_B\in H(\nabla\wedge; B)|\ B\ \ \mbox{is any bounded subdomain of $X$} \}
\]
and
\[
H(\nabla\wedge; B)=\{ U\in (L^2(B))^3|\ \nabla\wedge U\in (L^2(B))^3 \}.
\]
Clearly, the scattering amplitude $A_\infty$ depends also on the underlying passive EM medium $(\Omega; \tilde\varepsilon, \tilde\mu, \tilde\sigma)$ and the active electric current $\widetilde J$, and hence we 
shall write $A_\infty(\hat{x}; E^i, (\Omega; \tilde\varepsilon,\tilde\mu, \tilde\sigma),$ $ \widetilde J)$ to emphasize such dependence if necessary. 
An important {\it inverse scattering problem} arising from practical applications is to recover 
the medium $(\Omega; \tilde\varepsilon,\tilde\mu, \tilde\sigma)$ and/or the current 
$\widetilde J$ by knowledge of $A_\infty(\hat{x}; E^i)$. This inverse problem is of fundamental importance to many areas of science and technology, such as radar and sonar, geophysical exploration, non-destructive testing, and medical imaging. We refer the readers to \cite{KSU} \cite{18} \cite{19} and the references therein for the studies on uniqueness and stability of this inverse problem. In the present work, we are mainly concerned with the invisibility cloaking.
\begin{defn}\label{defn:cloaking}
Consider an EM medium as described in 
\eqref{eq:EM medium}, where 
$(D;\tilde\varepsilon_a, \tilde\mu_a, \tilde\sigma_a)$ and 
$(\Omega\backslash\overline{D};\varepsilon_c, \mu_c, \sigma_c)$ are 
the target and designed cloaking EM mediums respectively, and 
$\widetilde J\in L^2(\Omega)^3$ is an active object in $\Omega$. 
The medium 
$(\Omega; \tilde\varepsilon, \tilde\mu, \tilde\sigma)$ is called an (ideal) invisibility cloaking device if
no scattered fields are generated outside $\Om$, or equivalently
\[
A_\infty(\hat{x}; E^i, (\Omega; \tilde\varepsilon, \tilde\mu, \tilde\sigma), \widetilde J)=0.
\]
\end{defn}
Based on Definition~\ref{defn:cloaking}, the designed cloaking medium $(\Omega\backslash\overline{D}; \varepsilon_c, \mu_c, \sigma_c)$ makes the target medium $(D;\tilde\varepsilon_a, \tilde\mu_a, \tilde\sigma_a)$ and the active/radiating source $\widetilde J$ invisibile to the exterior EM detectors.  
From a practical point of view, 
the target medium and the electric current, $(D;\tilde\varepsilon_a, \tilde\mu_a, \tilde\sigma_a)$ and 
$\widetilde{J}$, should be allowed to be arbitrary for a cloaking device. This viewpoint would be adopted for our subsequent construction and investigation of the near-cloaking device. By the unique continuation principle for Maxwell's equations (cf. \cite{ColKre}), it is readily seen that for an ideal invisibility cloaking device, the scattered EM wave fields are completely trapped inside the device. The ideal invisibility cloaking of generic passive mediums was investigated in \cite{GKLU3,PenSchSmi}, and it turns out that one has to implement singular EM mediums. Indeed, the ideal invisibility constructions for the Maxwell equations proposed in \cite{GKLU3,PenSchSmi} make use of cloaking mediums $(\Omega\backslash\overline{D}; \varepsilon_c, \mu_c)$ which violate the regular conditions \eqref{eq:uniform elliptic}. Furthermore, it is shown in \cite{GKLU3} that if one intends to ideally cloak an active current, in addition to the singular cloaking medium, one needs to implement a special singular double coating 
to defeat the blow-up of the EM fields within the cloaked region.
The singular mediums present a great challenge for both theoretical analysis and practical fabrications. 
Several regularized constructions have been developed to avoid the singular structures. 
A truncation of singularities has been introduced in \cite{GKLUoe,GKLU_2,RYNQ}, 
and the `blow-up-a-point' transformation from \cite{GLU2,Leo,PenSchSmi} has been regularized to become 
a `blow-up-a-small-region' transformation in \cite{KOVW,KSVW,Liu}. By incorporating regularization into the cloaking construction, instead of the ideal/perfect invisibility, one considers the 
approximate/near invisibility; that is, to build up a regularized cloaking device
so that the resulting scattering amplitude is nearly negligible 
in terms of an asymptotically small regularization parameter $\rho\in\mathbb{R}_+$. This is the central focus of the current paper. For that purpose, we shall adopt the blow-up-a-small-region
strategy in the present study. Nevertheless, 
the truncation-of-singularity construction and the blow-up-a-small-region construction are equivalent to each other, as pointed out in \cite{KocLiuSunUhl}. 
Hence, all of the results obtained in this work hold equally for the truncation-of-singularity construction.

Due to its practical importance, the approximate cloaking has recently been extensively studied. In \cite{KSVW,Ammari1}, approximate cloaking schemes were developed for 
electric impedance tomography which can be regarded as optics at zero frequency. In \cite{Ammari2,Ammari3,KOVW,LiLiuSun,LiuSun,Liu}, several near-cloaking schemes were presented for scalar waves governed by the Helmholtz equation. On the contrary, not much 
has been done yet for the approximate cloaking of the full Maxwell equations. In \cite{LiuZhou}, the approximate cloaking was developed for the full Maxwell equations, where the near-cloaking construction is composed 
of three parts: a cloaked region $D^{(1)}$ containing the target medium $(\tilde\varepsilon_a, \tilde\mu_a, \tilde\sigma_a)$; a conducting layer $(D^{(2)}; \tilde\varepsilon_l, \tilde\mu_l, \tilde\sigma_l)$ 
located right outside the cloaked region $D^{(1)}$,  and a cloaking layer $(\Omega\backslash\overline{D}; \varepsilon_c^\rho, \mu_c^\rho)$ outside $D=D^{(1)}\cup D^{(2)}$, where $\rho\in\mathbb{R}_+$ is the regularizer and $(\varepsilon_c^\rho, \mu_c^\rho)$ degenerates to the singular cloaking medium in \cite{GKLU3,PenSchSmi} as $\rho\rightarrow 0$. The conducting layer $(D^{(2)}; \tilde\varepsilon_l, \tilde\mu_l, \tilde\sigma_l)$ between the cloaked and cloaking regions $D^{(1)}$ and $\Omega\backslash\overline{D}$ 
appears to be crucial to a practical near-cloaking construction. In fact, it is shown \cite{LiuZhou} that without the conducting layer, 
there always exist cloak-busting inclusions which defy any attempt to achieve the near-cloak, 
no matter how small the regularization parameter $\rho$ is. This reflects the highly unstable nature of the ideal invisibility cloaking with singular structures. However, the results of \cite{LiuZhou} were established only for the spherical geometry and the uniform cloaked content, namely 
both $\Omega$ and $D$ were assumed to be Euclidean balls and the medium parameters 
$\tilde{\varepsilon}_a, \tilde{\mu}_a$ and $\tilde{\sigma}_a$ were all constants multiple 
of the identity matrix.  Under these special settings, 
the Fourier-Bessel technique can be used to derive the analytic series expansions of the EM fields \cite{LiuZhou}, enabling one to assess the corresponding near-cloaking performance. Later, the study in \cite{LiuZhou} was generalized in \cite{BL} such that $\Omega$ and $D$ could be general smooth domains and the cloaked content could be an arbitrary regular passive medium. However, the conducting layer adopted in \cite{BL} for the cloaking construction is the same as the one in \cite{LiuZhou}, whose material tensors depend uniformly on the asymptotic parameter $\rho$ in a specific manner (see Remark~\ref{rem:conducting layer}). 

In this work, we investigate the near-cloaking devices with more general conducting layers. The material tensors of the conducting layers could be anisotropic, dependent on or independent of the regularization parameter $\rho$. 
On the one hand, this would extend the studies in the literature to an extremely general case, and on the other hand would be significant to practical applications when fabrication fluctuations occur. Moreover, 
only passive cloaked contents were studied for nearly cloaking so far, not any active contents involved. 
We shall investigate the nearly cloaking of both passive and active contents. 
In assessing the near-cloaking performance, we derive some systematic and sharp asymptotic estimates 
of the scattering amplitude in terms of the regularization parameter $\rho$ and the material tensors of the conducting layer. Our estimates are independent of the passive/active contents being cloaked.
This implies that one could nearly cloak an arbitrary content. Furthermore, 
the estimates can provide some practical guidance in choosing an appropriate conducting layer 
to improve the near cloaking of active contents. In addition, we emphasize 
that the asymptotic estimates were given in terms of the boundary measurements
in \cite{BL,LiuZhou}, whereas the asymptotic estimates are derived in terms of the scattering measurements
in this work and the corresponding asymptotic analysis is more delicate and technical.   

In addition to the transformation-optics approach adopted in the present study, there are several other effective approaches in the literature to realize the near-cloaking, and we mention the one based on anomalous localized resonance \cite{Ammari0,MN} and 
another one based on special (object-dependent) coatings \cite{AE}. Finally, we also mention a recent interesting work in \cite{Ammari2013}, where the near-cloaking of 
a perfectly conducting obstacle was studied for the full Maxwell equations.

The rest of the paper is organized as follows. In Section 2, we present the construction of our near-cloaking device and state the main results in estimating the cloaking performance. Section 3 is devoted to the proof
of the main result. 

\section{Near-cloak construction and the main results}

In this section, we present the construction of our near-cloaking device and 
formulate the major results in assessing the corresponding cloaking performance. 

Let $D$ and $\Omega$ be two bounded simply connected smooth domains in $\mathbb{R}^3$ such that $D\Subset\Omega$ and $D$ contains the origin. 
For $\rho\in\mathbb{R}_+$, we set 
\[
D_\rho:=\{\rho x; \ x\in D\}\,.
\]
Let $0<\rho<1$ be a small parameter. Assume that there exists an orientation-preserving and bi-Lipschitz mapping $F_\rho: \overline{\Omega}\backslash D_\rho\rightarrow \overline{\Omega}\backslash D$ such that
\begin{equation}\label{eq:map1}
F_\rho(\overline{\Omega}\backslash D_\rho)= \overline{\Omega}\backslash D
\quad \mbox{and} \quad F_\rho|_{\partial\Omega}=\mbox{Identity}.
\end{equation}
Now we define a transformation $F$ by 
\begin{equation}\label{eq:map whole}
F(x)=\begin{cases}
x,\quad & x\in \mathbb{R}^3\backslash\overline{\Omega},\\
F_\rho(x),\quad & x\in \Omega\backslash\overline{D}_\rho,\\
\frac{x}{\rho},\quad & x\in D_\rho\, ,
\end{cases}
\end{equation}
and an EM medium inside $\Omega\backslash\overline{D}$ by 
\begin{equation}\label{eq:cloaking medium}
\varepsilon_c^\rho(x)=F_* \varepsilon_0(x),\quad \mu_c^\rho(x)=F_* \mu_0(x),\quad \sigma_c^\rho(x)=0
\end{equation}
for $x\in\Omega\backslash \overline{D}$. Here $F_*$ denotes the {\it push-forward} operator defined by
\begin{equation}\label{eq:pushforward}
F_*m(x):=\frac{DF(y)\cdot m(y)\cdot DF(y)^T}{\left|\mbox{det} (DF)(y) \right|}\bigg|_{y=F^{-1}(x)} ,
\quad x\in\Omega\backslash\overline{D} ,
\end{equation}
where $m(y)$ denotes an EM parameter in $\Omega\backslash\overline{D}_\rho$, such as $\varepsilon, \mu$ or $\sigma$, and $DF$ represents the Jacobian matrix of the transformation $F$. In the sequel, 
we may often write \eqref{eq:cloaking medium} as
\[
(\Omega\backslash\overline{D}; \varepsilon_c^\rho, \mu_c^\rho)=F_*(\Omega\backslash\overline{D}_\rho; \varepsilon_0, \mu_0):=(F(\Omega\backslash\overline{D}_\rho); F_*\varepsilon_0, F_*\mu_0).
\]
Similarly, we set
\begin{equation}\label{eq:conducting layer}
(D\backslash\overline{D}_{1/2}; \tilde\varepsilon_l, \tilde\mu_l, \tilde\sigma_l)=F_*(D_\rho\backslash\overline{D}_{\rho/2}; \varepsilon_l, \mu_l, \sigma_l),
\end{equation}
where $\varepsilon_l(x)$, $\sigma_l(x)$ and $\mu_l(x)$ are given by 
\begin{equation}\label{eq:lossy layer 2}
\begin{split}
\varepsilon_l(x)=\rho^{-r}\alpha(x/\rho),\quad \sigma_l(x)=\rho^{-s}\beta(x/\rho), \quad 
\mu_l(x)=\rho^{-t}\gamma, \quad   x\in D_\rho\backslash\overline{D}_{\rho/2},
\end{split}
\end{equation}
for $r, s, t\in\mathbb{R}$. 
Here $\alpha(x)=(\alpha^{ij}(x))$ and $\beta(x)=(\beta^{ij}(x))$
are the material tensors for a regular EM medium in $\overline{D}\backslash D_{1/2}$, and are assumed to satisfy 
\begin{equation}\label{eq:uniform regular condition}
c_0|\xi|^2\leq \sum_{i,j=1}^3 m_l^{ij}(x)\xi_i\xi_j\leq C_0 |\xi|^2
\quad\mbox{for}\ \forall\xi\in\mathbb{R}^3\ \mbox{and}\ \mbox{a.e.}~ x\in \overline{D}\backslash D_{1/2},
\end{equation}
for $m_l=\alpha$ or $\beta$.
$\gamma=(\gamma^{ij})$ is assumed to be of the form 
\begin{equation}\label{eq:gamma}
\gamma^{-1}=\eta\,(\delta^{ij}),
\end{equation}
where $\eta$ is a constant satisfying $c_0\leq \eta\leq C_0$. Now, we consider an EM medium distribution in $\mathbb{R}^3$ as follows:
\begin{equation}\label{eq:EM medium physical}
\mathbb{R}^3; \tilde{\varepsilon}_\rho, \tilde{\mu}_\rho, \tilde{\sigma}_\rho=\begin{cases}
\varepsilon_0, \mu_0, \sigma_0\qquad & \mbox{in\ \ $\mathbb{R}^3\backslash\overline{\Omega}$},\\
\varepsilon_c^\rho, \mu_c^\rho, \sigma_c^\rho\qquad & \mbox{in\ \ $\Omega\backslash \overline{D}$},\\
\tilde\varepsilon_l, \tilde\mu_l, \tilde\sigma_l\quad & \mbox{in\ \  $D\backslash\overline{D}_{1/2}$},\\
\tilde{\varepsilon}_a, \tilde{\mu}_a, \tilde{\sigma}_a\quad & \mbox{in\ \ $D_{1/2}$},
\end{cases}
\end{equation}
where $(\Omega\backslash\overline{D}; \varepsilon_c^\rho, \mu_c^\rho, \sigma_c^\rho)$ and $(D\backslash\overline{D}_{1/2};\tilde\varepsilon_l, \tilde\mu_l, \tilde\sigma_l )$ are given in \eqref{eq:cloaking medium} and \eqref{eq:conducting layer} respectively, and $(D_{1/2}; \tilde\varepsilon_a, \tilde\mu_a, \tilde\sigma_a)$ is an arbitrary regular EM medium. Associated with the EM medium distribution 
$(\mathbb{R}^3; \tilde\varepsilon_\rho, \tilde\mu_\rho, \tilde\sigma_\rho)$, 
the EM scattering due to the incident fields $(E^i, H^i)$ 
can be described by 
\begin{equation}\label{eq:Maxwell physical}
\begin{cases}
\displaystyle{\nabla\wedge\widetilde E_\rho-i\omega\tilde\mu_\rho\widetilde H_\rho=0}\qquad\qquad &\mbox{in\ \ $\mathbb{R}^3$,}\\
\displaystyle{\nabla\wedge\widetilde H_\rho+i\omega\left(\tilde\varepsilon_\rho+i\frac{\tilde\sigma_\rho}{\omega}\right)\widetilde E_\rho=\widetilde{J}}\quad &\mbox{in\ \ $\mathbb{R}^3$},\\
\displaystyle{\widetilde E_\rho^-=\widetilde E_\rho|_{\Omega},\quad \widetilde{E}_\rho^+=(\widetilde{E}_\rho-E^i)|_{\mathbb{R}^3\backslash\overline{\Omega}},}\\
\displaystyle{\widetilde H_\rho^-=\widetilde H_\rho|_{\Omega},\quad \widetilde{H}_\rho^+=(\widetilde{H}_\rho-H^i)|_{\mathbb{R}^3\backslash\overline{\Omega}},}\\
\displaystyle{\lim_{|x|\rightarrow\infty}|x|\Big| (\nabla\wedge\widetilde E_\rho^+)(x)\wedge\frac{x}{|x|}-i\omega \widetilde E_\rho^+(x) \Big|=0,}
\end{cases}
 \end{equation}
 where $\widetilde J\in L^2(D)^3$ denotes an electric current in $D$. We shall assume that
\begin{equation}\label{eq:current assumption}
(\tilde\sigma_a(x)\xi)\cdot\xi \geq c_0|\xi|^2 \quad\mbox{for\ $\forall\xi\in\mathbb{R}^3$\ and a.e. $x\in supp(\widetilde J)\cap D_{1/2}$}.
\end{equation}
We refer to \eqref{eq:Maxwell physical} as the scattering problem in the physical space.

We are now in a position to state the main result of this paper.
\begin{thm}\label{thm:main}
Let $(\mathbb{R}^3; \tilde{\varepsilon}_\rho, \tilde{\mu}_\rho, \tilde{\sigma}_\rho)$ be the passive EM medium described by 
\eqref{eq:cloaking medium}--\eqref{eq:gamma}, \eqref{eq:EM medium physical}
and \eqref{eq:current assumption},  
$\widetilde{J}\in L^2(D)^3$ be an active current in $D$, and $\zeta_1$, $\zeta_2$ 
be the parameters given by 
\begin{align}
\zeta_1:=& \min\bigg(s+1, s+5-2(t+r), 5-2t-s\bigg),\label{eq:zeta1}\\
\zeta_2:=& \min\bigg( s, s+2-t-r, 2-t \bigg ).\label{eq:zeta2}
\end{align}
Assume $r, s, t\in \mathbb{R}$ are chosen such that $\zeta_1>0$. 
 Let $\widetilde A^\rho_\infty(\hat{x}):=A_\infty(\hat{x}; E^i, (\Omega; \tilde\varepsilon_\rho, \tilde\mu_\rho, \tilde\sigma_\rho), \widetilde J)$ be the scattering amplitude corresponding to $\widetilde E_\rho^+$ in \eqref{eq:Maxwell physical}. Then there exists a positive constant $\rho_0$ such that for any $\rho<\rho_0$, 
\begin{equation}\label{eq:main estimate physical}
\begin{split}
 &\|\widetilde A^\rho_\infty(\hat{x}; E^i)\|_{L^\infty(\mathbb{S}^2)}\\
\leq & C\Big( \rho^{\min(\zeta_1, 3)} \|E^i\|_{H(\nabla\wedge;\Omega)}
 + \rho^{\frac{\zeta_1}{2}}  \|\widetilde{J}\|_{L^2(D_{1/2})^3}+\rho^{\zeta_2}\|\widetilde J \|_{L^2(D\backslash D_{1/2})^3}  \Big).
\end{split}
\end{equation}
where $C$ is a positive constant depending only on $\alpha, \beta, \gamma, \omega$, $c_0$ in \eqref{eq:current assumption}, $C_0$ in \eqref{eq:uniform regular condition} and $\Omega, D$, but independent of 
$\rho, r, s, t$ and $\tilde\varepsilon_a, \tilde\mu_a, \tilde\sigma_a$, $\widetilde J$, $E^i$.
\end{thm}

The proof of Theorem~\ref{thm:main} will be given in Section~\ref{sect:3}. In the rest of this section, we 
give some remarks about the implications and practical significance of Theorem~\ref{thm:main} to the approximate invisibility cloaking. 

\begin{rem}\label{rem:near cloak}
By Theorem~\ref{thm:main}, it is readily seen that \eqref{eq:EM medium physical} yields a near-invisibility cloak, which is capable of nearly cloaking a passive medium 
$(D_{1/2}; \tilde\varepsilon_a, \tilde\mu_a, \tilde\sigma_a)$, 
an active current in both $D_{1/2}$ and $D\backslash D_{1/2}$, 
with an accuracy of orders $\rho^{\min(\zeta_1, 3)}$, 
$\rho^{\zeta_1/2}$, and $\rho^{\zeta_2}$ respectively. 
We note that $\zeta_2$ is required to be positive in order to achieve the cloaking effect, but 
Theorem~\ref{thm:main} will be proved without this requirement. 
Hence, the estimate \eqref{eq:main estimate physical} is rather general in 
this sense. 
The estimate \eqref{eq:main estimate physical} is independent of the passive medium $(D_{1/2}; \tilde\varepsilon_a, \tilde\mu_a, \tilde\sigma_a)$ and the active current $\widetilde{J}$, so 
the contents being cloaked could be arbitrary. 
Clearly, this is of significant importance for a near-cloaking device in applications. 
We mention that 
the cloaking of active contents was studied in \cite{GKLU3}, where 
the authors considered the ideal cloaking by employing the singular cloaking medium $(\varepsilon_c^\rho, \mu_c^\rho)$ in the theoretic limiting case $\rho=+0$. 
However, it was shown there that one cannot cloak an active content by merely
using $(\varepsilon_c^\rho, \mu_c^\rho)$ in the theoretic limiting case $\rho=+0$, 
otherwise one would have the blow-up of the EM fields within the cloaked region. 
Theorem~\ref{thm:main} indicates that our near-cloaking construction~\eqref{eq:EM medium physical} is much more stable, even in cloaking active contents.     
\end{rem}

\begin{rem}\label{rem:conducting layer}
By \eqref{eq:conducting layer} and \eqref{eq:lossy layer 2}, it is straightforward to show that in the physical space,
\begin{equation}\label{eq:physical lossy layer}
\tilde\varepsilon_l(x)=\rho^{1-r}\alpha(x),\ \ \tilde\sigma_l(x)=\rho^{1-s}\beta(x),\ \ \tilde\mu_l(x)=\rho^{1-t}\gamma(x),\quad x\in D\backslash\overline{D}_{1/2}.
\end{equation}
Hence, if we take $r=t=0$, $s=2$, and $\alpha=\beta=\gamma=C_0(\delta^{ij})$ with $C_0$ being 
a positive constant, we obtain the conducting layer employed in \cite{BL,LiuZhou}. 
In this case we have $\min(\zeta_1, 3)=3$, 
hence Theorem~\ref{thm:main} recovers the results in \cite{BL,LiuZhou} in near-cloaking passive mediums within an accuracy of order  $\rho^3$. It is interesting to note that by taking $r=s=t=1$, the conducting layer \eqref{eq:physical lossy layer} is independent of the asymptotic parameter $\rho$, and the estimate \eqref{eq:main estimate physical} reduces to 
\begin{equation}\label{eq:main estimate physical 2}
 |\widetilde A^\rho_\infty(\hat{x}; E^i)|
\leq C\Big( \rho^2 \|E^i\|_{H(\nabla\wedge;\Omega)}+ \rho  \|\widetilde{J}\|_{L^2(D)^3}\Big ).
\end{equation}
That is, by employing a regular conducting layer without relating to 
the regularization parameter $\rho$, one could achieve a near-invisibility cloak which is capable of cloaking a passive content  
and an active content  with an accuracy of order $\rho^2$ and $\rho$ 
respectively. On the other hand, we emphasize that our incorporation of 
the anisotropic parameters $\alpha$ and $\beta$ is of significant interests in applications 
if any fabrication fluctuation occurs. Moreover, our general estimate would provide a guideline 
for practically choosing the conducting layer to produce customized near-cloaking effects. For instance, 
if we take $r=0$, $t=-s$ with $s\in\mathbb{R}_+$, then one can check that the larger the index $s$ is, 
the better accuracy of near-cloaking the current $\widetilde{J}$ that one can achieve. 
\end{rem}

\section{Proof of the major result}\label{sect:3}

This section is devoted to the proof of Theorem~\ref{thm:main}, the major result of this work. 
We first collect some important function spaces that are needed for the subsequent analysis. 

\subsection{Function spaces}

Let $\Gamma$ be the smooth boundary of a bounded domain in $\mathbb{R}^3$, with 
$\nu$ being its outward unit normal vector. It is known that $H^s(\Gamma)$ is well-defined 
for $|s|\leq 2$ (cf. \cite{Gri}, \cite{Lio}). By $TH^s(\Gamma)$ we denote the subspace of all the functions 
$U\in H^s(\Gamma)^3$, which are orthogonal to the unit outward normal vector $\nu$. 
For $|s|\leq 2$, we can decompose a $U\in H^s(\Gamma)^3$ into a sum of the form 
$U=U_t+\nu U_\nu$, where $U_t$ and $U_\nu$ are the tangential and normal components,   
i.e., $U_t=-\nu\wedge(\nu\wedge U)$, $U_\nu=\langle \nu, U\rangle$. This gives rise to a decomposition 
of $H^s(\Gamma)^3$ for $|s|\leq 2$: $H^s(\Gamma)^3=TH^s(\Gamma)\bigoplus NH^s(\Gamma)$.  
Since $\Gamma$ is smooth, we know $TH^s(\Gamma)$ coincides with $\nu\wedge H^s(\Gamma)^3$. 
Let $\text{Div}$ be the surface divergence operator on $\Gamma$, then 
we will frequently use  in the sequel the following dual space of $TH^{1/2}(\Gamma)$:
\[
TH_{\text{Div}}^{-1/2}(\Gamma)=\left\{U\in TH^{-1/2}(\Gamma)|\ \text{Div}(U)\in H^{-1/2}(\Gamma)\right\}\;,
\]
and a skew-symmetric bilinear form
$\mathcal{B}$: $TH_{\text{Div}}^{-1/2}(\Gamma)\wedge TH_{\text{Div}}^{-1/2}(\Gamma)
\rightarrow \mathbb{C}$, given by 
the non-degenerate duality product  (cf.\,\cite{CosLou}):
\begin{equation}\label{eq:duality}
\mathcal{B}(\mathbf{j},\mathbf{m})=\int_{\Gamma} \mathbf{j}\cdot (\mathbf{m}\wedge \nu)\ ds, \quad 
\forall\, \mathbf{j}, \mathbf{m} \in TH_{\text{Div}}^{-1/2}(\Gamma)\,.
\end{equation}

\subsection{Proof of Theorem~\ref{thm:main}}

We first present a lemma with some key ingredients of the transformation optics, whose proof 
is available in \cite{LiuZhou}.

\begin{lem}\label{thm:trans opt}
Let $(\Omega; \varepsilon, \mu, \sigma)$ be a regular EM medium, $J\in L^2(\Omega)^3$ be
a current in $\Om$, and $x'=\mathcal{F}(x): \Omega\rightarrow\Omega$ be 
a bi-Lipschitz and orientation-preserving mapping, whose restriction on ${\partial\Omega}$ 
is the Identity. 
Suppose that $E, H\in H(\nabla\wedge;\Omega)$ are the EM fields satisfying
\[
\begin{split}
\nabla\wedge E-i\omega \mu H=0\qquad & \mbox{in\ \ $\Omega$,}\\
\nabla\wedge H+i\omega \left( \varepsilon+i\frac{\sigma}{\omega} \right) E=J\qquad & \mbox{in\ \ $\Omega$},
\end{split}
\]
If we define the {\emph pull-back fields} by
\[
\begin{split}
E'&=(\mathcal{F}^{-1})^* E:=(D\mathcal{F})^{-T}E\circ \mathcal{F}^{-1},\\
H'&=(\mathcal{F}^{-1})^* H:=(D\mathcal{F})^{-T}H\circ \mathcal{F}^{-1},\\
J'&=(\mathcal{F}^{-1})^* J:=\frac{1}{|\text{\emph{det}}(D\mathcal{F})|} (D\mathcal{F}) J\circ \mathcal{F}^{-1}\,,
\end{split}
\]
then the pull-back fields $E', H'\in H(\nabla'\wedge;\Omega)$ 
satisfy the following Maxwell equations
\[
\begin{split}
\nabla' \wedge E'-i\omega \mu' H'=& 0\qquad \mbox{in\ \ $\Omega$},\\
\nabla' \wedge H'+i\omega \left(\varepsilon'+i\frac{\sigma'}{\omega}\right) E'=& J'\quad\ \ \mbox{in\ \ $\Omega$},
\end{split}
\]
where $\nabla'\wedge$ denote the {\em curl} operator in the $x'$-coordinates, $\varepsilon'$, $\mu'$ and $\sigma'$ are the push-forwards of $\varepsilon, \mu$ and $\sigma$ via $\mathcal{F}$, i.e., 
$(\Omega; \varepsilon',\mu',\sigma')=\mathcal{F}_*(\Omega; \varepsilon, \mu, \sigma)$. Moreover, it holds that 
\[
\nu\wedge E'=\nu\wedge E,\quad \nu\wedge H'=\nu\wedge H\quad\mbox{on\ \ $\partial\Omega$}.
\]
\end{lem}

Next, for the EM fields $(\widetilde{E}_\rho, \widetilde{H}_\rho)$ described by 
\eqref{eq:Maxwell physical} associated with the physical scattering problem, 
we define 
\begin{equation}\label{eq:EM virtual 1}
E_\rho=F^* \widetilde{E}_\rho\quad \mbox{and}\quad H_\rho=F^*\widetilde{H}_\rho,
\end{equation}
where $F$ is the transformation given by \eqref{eq:map whole}. They by Lemma~\ref{thm:trans opt} it is straightforward to verify that the two fields 
$E_\rho, H_\rho\in H_{loc}(\nabla\wedge; \mathbb{R}^3)$, 
and satisfy the following system
\begin{equation}\label{eq:Maxwell virtual}
\begin{cases}
\displaystyle{\nabla\wedge E_\rho-i\omega\mu_\rho H_\rho=0}\qquad\qquad &\mbox{in\ \ $\mathbb{R}^3$,}\\
\displaystyle{\nabla\wedge H_\rho+i\omega\left(\varepsilon_\rho+i\frac{\sigma_\rho}{\omega}\right) E_\rho=J}\quad &\mbox{in\ \ $\mathbb{R}^3$},\\
\displaystyle{ E_\rho^-= E_\rho|_{D_\rho},\quad E_\rho^+=(E_\rho-E^i)|_{\mathbb{R}^3\backslash\overline{D}_\rho},}\\
\displaystyle{ H_\rho^-= H_\rho|_{D_\rho},\quad H_\rho^+=(H_\rho-H^i)|_{\mathbb{R}^3\backslash\overline{D}_\rho},}\\
\displaystyle{\lim_{|x|\rightarrow\infty}|x|\Big| (\nabla\wedge E_\rho^+)(x)\wedge\frac{x}{|x|}-i\omega E_\rho^+(x) \Big|=0,}
\end{cases}
 \end{equation}
where $J(x)$ and the EM medium $(\varepsilon_\rho, \mu_\rho, \sigma_\rho)$ are given by 
\begin{equation}\label{eq:current virtual}
J(x):=F^*\widetilde{J}(x)=\frac{1}{\rho^2}\widetilde J\left(\frac x \rho\right), \quad x\in D_\rho,
\end{equation}
and
\begin{equation}\label{eq:virtual medium}
\mathbb{R}^3; \varepsilon_\rho, \mu_\rho, \sigma_\rho=\begin{cases}
\varepsilon_0, \mu_0, \sigma_0\qquad &\mbox{in\ \ $\mathbb{R}^3\backslash \overline{D}_\rho$},\\
\varepsilon_l, \mu_l, \sigma_l \quad &\mbox{in\ \ $D_\rho\backslash\overline{D}_{\rho/2}$},\\
\varepsilon_a, \mu_a, \sigma_a\qquad &\mbox{in\ \ $D_{\rho/2}$},
\end{cases}
\end{equation}
with $(D_\rho\backslash\overline{D}_{\rho/2}; \varepsilon_l, \mu_l, \sigma_l)$ given in the form 
\eqref{eq:lossy layer 2}--\eqref{eq:gamma}, and
\begin{equation}\label{eq:virtual object}
(D_{\rho/2}; \varepsilon_a, \mu_a, \sigma_a):=(F^{-1})_*(D_{\rho/2}; \tilde\varepsilon_a, \tilde\mu_a ,\tilde\sigma_a).
\end{equation}
For our subsequent use, we note by \eqref{eq:map whole}, \eqref{eq:pushforward} 
and straightforward calculations that 
\begin{equation}\label{eq:virtual tensors}
m_a(x)=\rho^{-1}\widetilde m_a(\rho^{-1}x), \quad x\in D_{\rho/2},
\end{equation}
for $m=\varepsilon, \mu, \sigma$, hence it follows from \eqref{eq:current assumption} that 
\begin{equation}\label{eq:current assumption virtual}
(\sigma_a(x)\xi)\cdot\xi\geq c_0\rho^{-1}|\xi|^2,\quad\mbox{for}\ 
\forall\,\xi\in\mathbb{R}^3 \ \mbox{and}\  \mbox{a.e.}\ x\in supp(J)\cap D_{\rho/2}.
\end{equation}

Next we shall establish a series of lemmas which provide several crucial relations 
and estimates for the proof of Theorem~\ref{thm:main}. 

\begin{lem}\label{lem:1}
Let $B_R$ be a central ball of radius $R$ such that $D_\rho\Subset B_R$. Then 
the solutions $E_\rho, H_\rho \in H_{loc}(\nabla\wedge; \mathbb{R}^3)$ 
to the system \eqref{eq:Maxwell virtual} satisfy
\begin{equation}\label{eq:control l2}
\begin{split}
&\int_{D_\rho\backslash D_{\rho/2}}\sigma_l E_\rho^-\cdot\overline{E_\rho^-}\ dx+\int_{D_{\rho/2}}\sigma_a E_\rho^-\cdot\overline{E_\rho^-}\ dx\\
=& \Re\int_{\partial B_R} (\nu\wedge\overline{E_\rho^+})\cdot\left[ \nu\wedge(\nu\wedge H_\rho^+) \right ]\ ds_x
+\Re\int_{\partial B_R} (\nu\wedge\overline{E^i})\cdot\left[ \nu\wedge(\nu\wedge H_\rho^+) \right ]\ ds_x\\
+&\Re\int_{\partial B_R} (\nu\wedge\overline{E^+_\rho})\cdot\left[ \nu\wedge(\nu\wedge H^i) \right ]\ ds_x
+\Re\int_{D_\rho} J\cdot\overline{E_\rho^-}\ dx.
\end{split}
\end{equation}
\end{lem}

\begin{proof}
First of all, it is easy to see that 
the solutions $(E_\rho^\pm, H_\rho^\pm)$ to \eqref{eq:Maxwell virtual} satisfy
\begin{align}
\nabla\wedge E_\rho^-=-i\omega\mu_\rho H_\rho^-\qquad & \mbox{in\ \ $D_\rho$},\label{eq:l1}\\
\nabla\wedge H_\rho^-=-i\omega\left(\varepsilon_\rho+i\frac{\sigma_\rho}{\omega}\right)E_\rho^-\qquad & \mbox{in\ \ $D_\rho$},\label{eq:l2}\\
 \nabla\wedge E_\rho^+=i\omega H_\rho^+,\quad \nabla\wedge H_\rho^+=-i\omega E_\rho^+\qquad & \mbox{in\ \ $B_R\backslash\overline{D}_\rho$},\label{eq:l3}\\
\nu\wedge E_\rho^-=\nu\wedge E_\rho^++\nu\wedge E^i\ \qquad &\mbox{on\ \ $\partial D_\rho$},\label{eq:l4}\\
\nu\wedge H_\rho^-=\nu\wedge H_\rho^++\nu\wedge H^i\ \qquad &\mbox{on\ \ $\partial D_\rho$}.\label{eq:l5}
\end{align}
Using \eqref{eq:l3} and integrating by parts we can deduce 
\begin{equation}\label{eq:d1}
\begin{split}
&-i\omega\int_{B_R\backslash\overline{D}_\rho} E_\rho^+\cdot\overline{E_\rho^+}\ ds=\int_{B_R\backslash\overline{D}_\rho} (\nabla\wedge H_\rho^+)\cdot \overline{E_\rho^+}\ dx\\
=&\int_{B_R\backslash\overline{D}_\rho} H_\rho^+\cdot(\nabla\wedge\overline{E_\rho^+})\ d x-\int_{\partial(B_R\backslash\overline{D}_\rho)} (\nu\wedge\overline{E_\rho^+})\cdot H_\rho^+\ ds_x\\
=&-i\omega\int_{B_R\backslash\overline{D}_\rho} H_\rho^+ ds_x+\int_{\partial B_R}(\nu\wedge\overline{E_\rho^+})\cdot\left[\nu\wedge(\nu\wedge H_\rho^+) \right]\ ds_x\\
&\hspace*{3.3cm} -\int_{\partial D_\rho}(\nu\wedge\overline{E_\rho^+})\cdot\left[\nu\wedge(\nu\wedge H_\rho^+) \right]\ ds_x\,,
\end{split}
\end{equation}
while using \eqref{eq:l1}--\eqref{eq:l2} and integrating by parts, we can write 
\begin{equation}\label{eq:d2}
\begin{split}
&-\int_{D_\rho\backslash\overline{D}_{\rho/2}}i\omega\left(\varepsilon_l+i\frac{\sigma_l}{\omega} \right)E_\rho^-\cdot\overline{E_\rho^-}\ dx-\int_{{D}_{\rho/2}}i\omega\left(\varepsilon_a+i\frac{\sigma_a}{\omega} \right)E_\rho^-\cdot\overline{E_\rho^-}\ dx\\
=&\int_{D_\rho}(\nabla\wedge H_\rho^-)\cdot\overline{E_\rho^-}\ dx+\int_{D_\rho} J\cdot\overline{E_\rho^-}\ dx\\
=&\int_{D_\rho} H_\rho^-\cdot(\nabla\wedge\overline{E_\rho^-})\ dx-\int_{\partial D_\rho}(\nu\wedge \overline{E_\rho^-})\cdot H_\rho^- \ ds_x+\int_{D_\rho} J\cdot\overline{E_\rho^-}\ dx\\
=&\int_{D_\rho} H_\rho^-\cdot(-i\omega\mu_\rho H_\rho^-)\ dx+\int_{D_\rho} J\cdot\overline{E_\rho^-}\ dx\\
&+\int_{\partial D_\rho} (\nu\wedge \overline{E_\rho^-})\cdot\left[\nu\wedge(\nu\wedge H_\rho^-)\right]\ ds_x.
\end{split}
\end{equation}
Now by taking the real parts of both sides of \eqref{eq:d2}, we obtain 
\begin{equation}\label{eq:d3}
\begin{split}
&\int_{D_\rho\backslash\overline{D}_{\rho/2}}\sigma_l E_\rho^-\cdot\overline{E_\rho^-}\ dx+\int_{D_{\rho/2}}\sigma_a E_\rho^-\cdot\overline{E_\rho^-}\ dx\\
=&\Re\int_{\partial D_\rho} (\nu\wedge \overline{E_\rho^-})\cdot \left[\nu\wedge(\nu\wedge H_\rho^-)\right]\ ds_x
+\Re\int_{D_\rho} J\cdot\overline{E_\rho^-}\ dx\,.
\end{split}
\end{equation}
On the other hand, taking the real parts of both sides of \eqref{eq:d1}, then 
adding them to \eqref{eq:d3},  we arrive at 
\begin{equation}\label{eq:d4}
\begin{split}
&\int_{D_\rho\backslash\overline{D}_{\rho/2}}\sigma_l E_\rho^-\cdot\overline{E_\rho^-}\ dx+\int_{D_{\rho/2}}\sigma_a E_\rho^-\cdot\overline{E_\rho^-}\ dx\\
=&\Re\int_{\partial B_R}(\nu\wedge\overline{E_\rho^+})\cdot\left[\nu\wedge(\nu\wedge H_\rho^+) \right]\ ds_x+\Re\int_{D_\rho} J\cdot\overline{E_\rho^-}\ dx\\
-&\Re\int_{\partial D_\rho}(\nu\wedge\overline{E_\rho^+})\cdot\left[\nu\wedge(\nu\wedge H_\rho^+) \right]\ ds_x
+\Re\int_{\partial D_\rho} (\nu\wedge \overline{E_\rho^-})\cdot \left[\nu\wedge(\nu\wedge H_\rho^-)\right]\ ds_x\,.
\end{split}
\end{equation}
For the last two terms in \eqref{eq:d4}, we can use the transmission conditions \eqref{eq:l4}--\eqref{eq:l5} 
and integration by parts to write 
\begin{equation}\label{eq:d5}
\begin{split}
&\int_{\partial D_\rho}(\nu\wedge\overline{E_\rho^-})\cdot \left[\nu\wedge(\nu\wedge H_\rho^-)\right]\ ds_x
-\int_{\partial D_\rho}(\nu\wedge\overline{E_\rho^+})\cdot\left[\nu\wedge(\nu\wedge H_\rho^+) \right]\ ds_x\\
=&\int_{\partial D_\rho} (\nu\wedge\overline{E^i})\cdot\left[\nu\wedge(\nu\wedge H_\rho^+) \right]\ ds_x
+\int_{\partial D_\rho} (\nu\wedge\overline{E^+_\rho})\cdot\left[\nu\wedge(\nu\wedge H^i) \right]\ ds_x\\
&\qquad+\int_{\partial D_\rho} (\nu\wedge\overline{E^i})\cdot\left[\nu\wedge(\nu\wedge H^i) \right]\ ds_x\,,
\end{split}
\end{equation}
while the following holds for the first two terms in the RHS of \eqref{eq:d5},  
\begin{equation}\label{eq:d6}
\begin{split}
&\Re\int_{\partial D_\rho} (\nu\wedge\overline{E^i})\cdot\left[\nu\wedge(\nu\wedge H_\rho^+) \right]\ ds_x
+\Re\int_{\partial D_\rho} (\nu\wedge\overline{E^+_\rho})\cdot\left[\nu\wedge(\nu\wedge H^i) \right]\ ds_x\\
=&\Re\int_{\partial B_R} (\nu\wedge\overline{E^i})\cdot\left[\nu\wedge(\nu\wedge H_\rho^+) \right]\ ds_x
+\Re\int_{\partial B_R}(\nu\wedge\overline{E_\rho^+})\cdot\left[\nu\wedge(\nu\wedge H^i) \right]\ ds_x\,.
\end{split}
\end{equation}
In fact, we immediately derive by integration by parts that 
\begin{equation}\label{eq:d7}
\begin{split}
&-\int_{\partial D_\rho} (\nu\wedge\overline{E^i})\cdot\left[\nu\wedge(\nu\wedge H_\rho^+) \right]\ ds_x+\int_{\partial B_R}(\nu\wedge\overline{E^i})\cdot\left[ \nu\wedge(\nu\wedge H_\rho^+)\right]\ ds_x\\
=&\int_{B_R\backslash\overline{D}_\rho} (\nabla\wedge H_\rho^+)\cdot\overline{E^i}\ dx-\int_{B_R\backslash\overline{D}_\rho} H_\rho^+\cdot(\nabla\wedge\overline{E^i})\ dx\\
=&i\omega\int_{B_R\backslash\overline{D}_\rho}\left[ -E_\rho^+\cdot\overline{E^i}+H_\rho^
+\cdot\overline{H^i}\ dx \right]\,
\end{split}
\end{equation}
and 
\begin{equation}\label{eq:d8}
\begin{split}
&-\int_{\partial D_\rho}(\nu\wedge\overline{E_\rho^+})\cdot\left[ \nu\wedge(\nu\wedge H^i) \right]\ ds_x+\int_{\partial B_R}(\nu\wedge \overline{E_\rho^+})\cdot\left[\nu\wedge(\nu\wedge H^i)\right]\ ds_x\\
=& i\omega\int_{B_R\backslash\overline{D}_\rho}\left[-E^i\cdot\overline{E_\rho^+}+H^i\cdot\overline{H_\rho^+} \right]\ dx\,.
\end{split}
\end{equation}
Clearly, \eqref{eq:d6} is a direct consequence of \eqref{eq:d7}-\eqref{eq:d8}. 
For the last term in \eqref{eq:d5}, we can use the Maxwell equations~\eqref{eq:Maxwell incident} and 
integration by parts to obtain 
\begin{equation}\label{eq:d9}
\begin{split}
\int_{\partial D_\rho} (\nu\wedge\overline{E^i})\cdot\left[ \nu\wedge(\nu\wedge H^i) \right]\ ds_x
=&\int_{D_\rho}\left((\nu\wedge H^i)\cdot\overline{E^i}-H^i\cdot(\nabla\wedge\overline{E^i})\right)\ dx\\
=&i\omega\int_{D_\rho}\left(-|E^i|^2+ |H^i|^2 \right )\ dx.
\end{split}
\end{equation}

Now combining \eqref{eq:d4}-\eqref{eq:d6} with \eqref{eq:d9} gives \eqref{eq:control l2}, 
so completes the proof of Lemma~\ref{lem:1}.
\end{proof}

In order to reduce the concerned scattering problem in the whole space $\mathbb{R}^3$ to a bounded domain problem, we next introduce the following auxiliary Maxwell system,
\begin{equation}\label{eq:auxiliary}
\begin{cases}
\nabla\wedge E-i\omega\mu_0 H=0 &\mbox{in\ \ $\mathbb{R}^3\backslash\overline{B}_R$},\\
\nabla\wedge H+i\omega\varepsilon_0 E=0 &\mbox{in\ \ $\mathbb{R}^3\backslash\overline{B}_R$},\\
\displaystyle{\lim_{|x|\rightarrow+\infty}|x|\Big| (\nabla\wedge E)(x)\wedge\frac{x}{|x|}-i\omega E(x) \Big|=0. }
\end{cases}
\end{equation}
Associated with the system \eqref{eq:auxiliary}, we introduce a boundary operator $\Lambda$, 
which maps the tangential component of the electric field  
to the tangential component of the magnetic field:
\begin{equation}\label{eq:DtN}
\Lambda(\nu\wedge E|_{\partial B_R})=\nu\wedge H|_{\partial B_R}:\ TH_{\text{Div}}^{-1/2}(\partial B_R)\rightarrow TH_{\text{Div}}^{-1/2}(\partial B_R),
\end{equation}
where $E, H\in H_{loc}(\nabla\wedge; \mathbb{R}^3\backslash\overline{B}_R)$
are the unique solutions to \eqref{eq:auxiliary}. We choose $R$ such that
$D_\rho\Subset B_R\Subset \Omega$ and $\omega$ is not an interior EM eigenvalue in the sense that 
the following Maxwell equations have only the trivial solutions $\widetilde E=\widetilde H=0$:
\begin{equation}\label{eq:auxiliary2}
\begin{cases}
\nabla\wedge \widetilde E-i\omega\mu_0 \widetilde H=0 &\mbox{in\ \ ${B}_R$},\\
\nabla\wedge \widetilde H+i\omega\varepsilon_0 \widetilde E=0 &\mbox{in\ \ ${B}_R$},
\end{cases}
\end{equation}
if $\nu\wedge \widetilde E|_{\partial B_R}=0$ or $\nu\wedge\widetilde H|_{\partial B_R}=0$.
We know the boundary operator $\Lambda$ in (\ref{eq:DtN}) is continuous and invertible \cite{Ned}. 

The next two lemmas provide some crucial estimates of the solutions 
$E_\rho, H_\rho\in H_{loc}(\nabla\wedge; \mathbb{R}^3)$
to the system \eqref{eq:Maxwell virtual}.
\begin{lem}\label{lem:2}
The solutions $E_\rho$, $H_\rho$ to the system \eqref{eq:Maxwell virtual} admit the following estimate,
\begin{eqnarray}
\int_{D_\rho\backslash D_{\rho/2}} |E_\rho^-|^2\ dx
&\leq & C\rho^s\bigg\{ \|\nu\wedge E_\rho^+\|_{TH^{-1/2}_{\text{\emph{Div}}}(\partial B_R)} \|\Lambda(\nu\wedge E_\rho^+)\|_{TH^{-1/2}_{\text{\emph{Div}}}(\partial B_R)}\nb\\
&&+ \|\nu\wedge E^i\|_{TH^{-1/2}_{\text{\emph{Div}}}(\partial B_R)}\|\Lambda(\nu\wedge E_\rho^+)\|_{TH^{-1/2}_{\text{\emph{Div}}}(\partial B_R)}\nb\\
&&+ \|\nu\wedge H^i\|_{TH^{-1/2}_{\text{\emph{Div}}}(\partial B_R)}\|\nu\wedge E_\rho^+\|_{TH^{-1/2}_{\text{\emph{Div}}}(\partial B_R)}\bigg\}\nb\\
&&+C\rho^{2s-1}\|\widetilde J\|^2_{L^2(D\backslash D_{1/2})^3}+C\rho^{s}\|\widetilde J\|^2_{L^2(D_{1/2})^3}
\label{eq:control l2 2}
\end{eqnarray}
where $C$ is a constant depending only on $c_0$ in \eqref{eq:current assumption virtual}.
\end{lem}

\begin{proof}
Without loss of generality, we may assume that $supp(J)=D_{\rho/2}$. By using \eqref{eq:lossy layer 2}-\eqref{eq:uniform regular condition}, \eqref{eq:current assumption virtual}, \eqref{eq:duality} and the Cauchy-Schwartz inequality, we first deduce from \eqref{eq:control l2} that
\begin{equation}\label{eq:e1}
\begin{split}
& c_0\rho^{-s}\|E_\rho^-\|^2_{L^2(D_\rho\backslash D_{\rho/2})^3}+c_0\rho^{-1}\|E_\rho^-\|^2_{L^2(D_{\rho/2})^3}\\
\leq & \int_{D_\rho\backslash D_{\rho/2}}\sigma_l E_\rho^-\cdot\overline{E_\rho^-}\ dx+\int_{D_{\rho/2}}\sigma_a E_\rho^-\cdot\overline{E_\rho^-}\ dx\\
\leq & \bigg\{ \|\nu\wedge E_\rho^+\|_{TH^{-1/2}_{\text{{Div}}}(\partial B_R)} \|\Lambda(\nu\wedge E_\rho^+)\|_{TH^{-1/2}_{\text{{Div}}}(\partial B_R)}\\
&\qquad\ \ + \|\nu\wedge E^i\|_{TH^{-1/2}_{\text{{Div}}}(\partial B_R)}\|\Lambda(\nu\wedge E_\rho^+)\|_{TH^{-1/2}_{\text{{Div}}}(\partial B_R)}\\
&\qquad\ \ + \|\nu\wedge H^i\|_{TH^{-1/2}_{\text{{Div}}}(\partial B_R)}\|\nu\wedge E_\rho^+\|_{TH^{-1/2}_{\text{{Div}}}(\partial B_R)}\bigg\}
+\Big|\int_{D_\rho} J\cdot\overline{E_\rho^-}\ dx\Big|.
\end{split}
\end{equation}
For the last term above, it follows from the relation 
\[
\|\widetilde{J}(\frac{\cdot}{\rho})\|_{L^2(D_\rho)}=\rho^{3/2}\|\widetilde{J}(\cdot)\|_{L^2(D)}
\] 
and  \eqref{eq:current virtual}  that 
\begin{equation}\label{eq:e2}
\begin{split}
\Big| \int_{D_\rho} J\cdot\overline{E_\rho^-} \ dx \Big|
\leq & \|J\|_{L^2(D_\rho\backslash D_{\rho/2})^3}\|E_\rho^-\|_{L^2(D_\rho\backslash D_{\rho/2})^3}+\|J\|_{L^2(D_{\rho/2})^3}\|E_\rho^-\|_{L^2(D_\rho/2)^3}\\
\leq &\frac{\rho^s}{2c_0} \|\rho^{-2} \widetilde{J}(\frac{\cdot}{\rho})\|^2_{L^2(D_\rho\backslash D_{\rho/2})^3}+\frac{c_0\rho^{-s}}{2}\|E_\rho^-\|^2_{L^2(D_\rho\backslash D_{\rho/2})^3}\\
&+\frac{\rho}{2c_0}\|\rho^{-2}\widetilde{J}(\frac{\cdot}{\rho})\|^2_{L^2(D_{\rho/2})^3}+\frac{c_0\rho^{-1}}{2}\|E_\rho^-\|^2_{L^2(D_{\rho/2})^3}\\
=& \frac{\rho^{s-1}}{2c_0}\|\widetilde{J}\|^2_{L^2(D\backslash D_{1/2})}+\frac{c_0\rho^{-s}}{2}\|E_\rho^-\|^2_{L^2(D_\rho\backslash D_{\rho/2})^3}\\
&+\frac{1}{2c_0}\|\widetilde{J}\|^2_{L^2(D_{1/2})^3}+\frac{c_0\rho^{-s}}{2}\|E_\rho^-\|^2_{L^2(D_\rho\backslash D_{\rho/2})^3}\,,
\end{split}
\end{equation}
where the two terms involving $E_\rho^-$ can be estimated by using 
\eqref{eq:e1}-\eqref{eq:e2} as follows
\begin{equation}\label{eq:e3}
\begin{split}
& \frac{c_0\rho^{-s}}{2}\|E_\rho^-\|^2_{L^2(D_\rho\backslash D_{\rho/2})^3}+\frac{c_0\rho^{-1}}{2}\|E_\rho^-\|^2_{L^2(D_{\rho/2})^3}\\
\leq & \bigg\{ \|\nu\wedge E_\rho^+\|_{TH^{-1/2}_{\text{{Div}}}(\partial B_R)} \|\Lambda(\nu\wedge E_\rho^+)\|_{TH^{-1/2}_{\text{{Div}}}(\partial B_R)}\\
&\qquad\ \ + \|\nu\wedge E^i\|_{TH^{-1/2}_{\text{{Div}}}(\partial B_R)}\|\Lambda(\nu\wedge E_\rho^+)\|_{TH^{-1/2}_{\text{{Div}}}(\partial B_R)}\\
&\qquad\ \ + \|\nu\wedge H^i\|_{TH^{-1/2}_{\text{{Div}}}(\partial B_R)}\|\nu\wedge E_\rho^+\|_{TH^{-1/2}_{\text{{Div}}}(\partial B_R)}\bigg\}\\
&+\frac{\rho^{s-1}}{2c_0}\|\widetilde{J}\|^2_{L^2(D\backslash D_{1/2})}+\frac{1}{2c_0}\|\widetilde{J}\|^2_{L^2(D_{1/2})^3},
\end{split}
\end{equation}
which, along with (\ref{eq:e1}), implies \eqref{eq:control l2 2} immediately.
\end{proof}

\begin{lem}\label{lem:bc control}
The following estimate holds for the solutions $(E_\rho, H_\rho)$ to the system 
\eqref{eq:Maxwell virtual}:
\begin{eqnarray}
\left\| (\nu\wedge E_\rho^-) (\rho\ \cdot)\right\|_{TH^{-1/2}(\partial D)}
&\leq & C \rho^{\frac{\zeta_1}{2}-2} \bigg\{ \|\nu\wedge E_\rho^+\|^{1/2}_{TH^{-1/2}_{\text{\emph{Div}}}(\partial B_R)} \|\Lambda(\nu\wedge E_\rho^+)\|^{1/2}_{TH^{-1/2}_{\text{\emph{Div}}}(\partial B_R)}\nb\\
&&+ \|\nu\wedge E^i\|^{1/2}_{TH^{-1/2}_{\text{\emph{Div}}}(\partial B_R)}\|\Lambda(\nu\wedge E_\rho^+)\|^{1/2}_{TH^{-1/2}_{\text{\emph{Div}}}(\partial B_R)}\nb\\
&&+ \|\nu\wedge H^i\|^{1/2}_{TH^{-1/2}_{\text{\emph{Div}}}(\partial B_R)}\|\nu\wedge E_\rho^+\|^{1/2}_{TH^{-1/2}_{\text{\emph{Div}}}(\partial B_R)}\bigg\}\nb\\
&&+C\rho^{\frac{\zeta_1}{2}-2}\|\widetilde J\|_{L^2(D_{1/2})^3}+C\rho^{\zeta_2-2}\|\widetilde J \|_{L^2(D\backslash D_{1/2})^3}, \label{eq:bc control 1}
\end{eqnarray}
where $\zeta_1$ and $\zeta_2$ are given in \eqref{eq:zeta1}--\eqref{eq:zeta2}, and $C$ is a positive constant dependent only on $D$,$\Omega$ and $c_0$ in \eqref{eq:current assumption virtual}, but independent of $E^i, H^i$, $\widetilde J$ and $\rho$.
\end{lem}

\begin{proof}
It suffices to show that the same estimate in \eqref{eq:bc control 1} holds for $\left\| (\nu\wedge E_\rho^-) (\rho\ \cdot)\right\|_{H^{-1/2}(\partial D)^3}$. We shall make use of the following duality identity
\begin{equation}\label{eq:duality}
\begin{split}
&\left\| (\nu\wedge E_\rho^-)(\rho\ \cdot) \right\|_{H^{-1/2}(\partial D)^3}
=\sup_{\|\phi\|_{H^{1/2}(\partial\Omega)^3}\leq 1}\Big| \int_{\partial D} (\nu\wedge E_\rho^-)(\rho x)\cdot \phi(x)\ ds_x \Big|.
\end{split}
\end{equation}
By Lemma~\ref{lem:extension} given in the following, for any $\phi\in H^{1/2}(\partial D)^3$, there exists $U\in H^2(D)^3$ such that

\begin{enumerate}
\item[(i)]~~$\nu\wedge U=0$ on $\partial D$,
\item[(ii)]~~$\nu\wedge(\nu\wedge(\nabla\wedge U))=\nu\wedge(\nu\wedge\phi)$ on $\partial D$,
\item[(ii)]~~$\|U\|_{H^2(D)^3}\leq C \|\phi\|_{H^{1/2}(\partial D)^3}$,
\item[(iv)]~~$U=0$ in $D_{1/2}$.
\end{enumerate}
For $y\in D_\rho$, we let
$
x:={y}/{\rho}\in D
$, 
and
\[
E(x):=E_\rho^-(\rho x)=E_\rho^-(y),\ \ H(x):=H_\rho^-(\rho x)=H_\rho^-(y)\,. 
\]
Then using \eqref{eq:gamma}, we can compute as follows
\begin{eqnarray}
&& \int_{\partial D} (\nu\wedge E_\rho^-)(\rho x)\cdot \phi(x)\ ds_x \label{eq:bc control 3}\\
&=& -\int_{\partial D} \eta^{-1} (\nu\wedge E)(x)\cdot (\nu\wedge (\nu\wedge(\gamma^{-1}\nabla\wedge U)))(x)\ ds_x\nb\\
&=&\int_{\partial D}(\nu\wedge(\gamma^{-1}\nabla\wedge U))(x)\cdot \eta^{-1} E(x)\ ds_x
-\int_{\partial D}(\nu\wedge(\gamma^{-1}\nabla\wedge E))(x)\cdot \eta^{-1} U(x)\ ds_x\nb\\
&=&\int_{D}(\nabla\wedge(\gamma^{-1}\nabla\wedge U))(x)\cdot \eta^{-1}E(x)\ dx
-\int_{D}(\nabla\wedge(\gamma^{-1}\nabla\wedge E))(x)\cdot \eta^{-1}U(x)\ dx.\nb
\end{eqnarray}
It follows from \eqref{eq:lossy layer 2} and \eqref{eq:current virtual} that 
\begin{equation}\label{eq:control 4}
\begin{split}
\nabla_y\wedge E_\rho^-(y)=&\,i\omega \mu_l(y) H_\rho^-(y),\\
\nabla_y\wedge H_\rho^-(y)=&-i\omega\left(\varepsilon_l(y)+i\frac{\sigma_l(y)}{\omega}\right) E_\rho^-(y)+J(y),
\end{split}
\end{equation}
for $y\in D_\rho\backslash\overline{D}_{\rho/2}$. Then it is straightforward to verify 
for $x\in D\backslash\overline{D}_{1/2}$ that
\begin{equation}\label{eq:control 5}
\begin{split}
\nabla_x\wedge E(x)=&i\omega \rho^{1-t}\gamma H(x),\\
\nabla_x\wedge H(x)=&-i\omega \left( \rho^{1-r}\alpha(x)+i\rho^{1-s}\frac{\beta(x)}{\omega}\right) E(x)+\rho^{-1}\widetilde{J}(x)\,, 
\end{split}
\end{equation}
and 
\begin{equation}\label{eq:control 6}
\begin{split}
& \nabla_x\wedge(\gamma^{-1}(x)\nabla_x\wedge E(x))
=\omega^2\left( \rho^{2-t-r}\alpha(x)+i\rho^{2-t-s}\frac{\beta(x)}{\omega} \right) E(x)+i\omega \rho^{-t}\widetilde{J}(x).
\end{split}
\end{equation}
By combining \eqref{eq:bc control 3} with \eqref{eq:control 6}, we obtain 
\begin{eqnarray}
&& \int_{\partial D}(\nu\wedge E_\rho^-)(\rho x)\cdot \phi(x)\ ds_x\nb\\
&=&\eta^{-1}\int_{D\backslash\overline{D}_{1/2}}\bigg[ (\nabla\wedge(\gamma^{-1}\nabla\wedge U))(x)
-\omega^2\left( \rho^{2-t-r}\alpha(x)+i\rho^{2-t-s}\frac{\beta(x)}{\omega}\right) U(x) \bigg]\cdot E(x)dx\nb\\
&& -i\omega\eta^{-1}\rho^{-t}\int_{D\backslash\overline{D}_{1/2}} \widetilde{J}(x)\cdot U(x)\ dx.
\label{eq:control 7}
\end{eqnarray}
This immediately yields 
\begin{equation}\label{eq:control 8}
\begin{split}
& \left| \int_{\partial D} (\nu\wedge E_\rho^-)(\rho x)\cdot\phi(x) \right |\\
\leq & C \rho^\theta \|E\|_{L^2(D\backslash D_{1/2})}\| U \|_{H^2(D\backslash \overline{D}_{1/2})}+C \rho^{-t} \|\widetilde J\|_{L^2(D\backslash D_{1/2})}\|U\|_{H^2(D\backslash\overline{D}_{1/2})}\\
\leq & C\left(\rho^{-3/2+\theta}\|E_\rho^-\|_{L^2(D_\rho\backslash D_{\rho/2})}+\rho^{-t}\|\widetilde J\|_{L^2(D\backslash D_{1/2})}\right) \|\phi\|_{H^{1/2}(\partial D)^3},
\end{split}
\end{equation}
where 
$\theta=\min(0, 2-t-r, 2-t-s )$,
and $C$ is a positive constant depending on $\alpha, \beta, \gamma, \omega$ and $D$, but independent of $\phi$, $\widetilde J$, $E_\rho^-$, $\rho$. Then by (\ref{eq:duality}) we know 
from \eqref{eq:control 8} that 
\begin{equation*}\label{eq:control 9}
\begin{split}
&\|(\nu\wedge E_\rho^-)(\rho\ \cdot)\|_{H^{-1/2}(\partial D)^3}
\leq  
C\left( \rho^{-3/2+\theta}\| E_\rho^- \|_{L^2(D_\rho\backslash D_{\rho/2})^3}
+\rho^{-t}\| \widetilde{J} \|_{L^2(D \backslash D_{1/2})^3} \right).
\end{split}
\end{equation*}
Finally, by means of the estimates \eqref{eq:control 9} and \eqref{eq:control l2 2} 
we can directly show the existence of two generic constants $C_1$ and $C_2$ such that
\begin{equation*}\label{eq:control 10}
\begin{split}
&\|(\nu\wedge E_\rho^-)(\rho\ \cdot)\|_{H^{-1/2}(\partial D)^3}\\
\leq & C_1\rho^{-3/2+\theta}\bigg\{\rho^{s/2} \bigg[ \|\nu\wedge E_\rho^+\|_{TH^{-1/2}_{\text{{Div}}}(\partial B_R)} \|\Lambda(\nu\wedge E_\rho^+)\|_{TH^{-1/2}_{\text{{Div}}}(\partial B_R)}\\
&\qquad\ \ + \|\nu\wedge E^i\|_{TH^{-1/2}_{\text{{Div}}}(\partial B_R)}\|\Lambda(\nu\wedge E_\rho^+)\|_{TH^{-1/2}_{\text{{Div}}}(\partial B_R)}\\
&\qquad\ \ + \|\nu\wedge H^i\|_{TH^{-1/2}_{\text{{Div}}}(\partial B_R)}\|\nu\wedge E_\rho^+\|_{TH^{-1/2}_{\text{{Div}}}(\partial B_R)}\bigg ]^{1/2}\\
&+\rho^{(2s-1)/2}\|\widetilde J\|_{L^2(D\backslash D_{1/2})^3}+\rho^{s/2}\|\widetilde J\|_{L^2(D_{1/2})^3}\bigg\}
 +\rho^{-t}\| \widetilde{J} \|_{L^2(D \backslash D_{1/2})^3}\\
\leq & C_2\rho^{s/2-3/2+\theta}\bigg\{ \|\nu\wedge E_\rho^+\|^{1/2}_{TH^{-1/2}_{\text{{Div}}}(\partial B_R)} \|\Lambda(\nu\wedge E_\rho^+)\|^{1/2}_{TH^{-1/2}_{\text{{Div}}}(\partial B_R)}\\
&\qquad\ \ + \|\nu\wedge E^i\|^{1/2}_{TH^{-1/2}_{\text{{Div}}}(\partial B_R)}\|\Lambda(\nu\wedge E_\rho^+)\|^{1/2}_{TH^{-1/2}_{\text{{Div}}}(\partial B_R)}\\
&\qquad\ \ + \|\nu\wedge H^i\|^{1/2}_{TH^{-1/2}_{\text{{Div}}}(\partial B_R)}\|\nu\wedge E_\rho^+\|^{1/2}_{TH^{-1/2}_{\text{{Div}}}(\partial B_R)}\bigg\}\\
&+C_2(\rho^{(2s-1)/2-3/2+\theta}+\rho^{-t})\|\widetilde J\|_{L^2(D\backslash D_{1/2})}
+C_2\rho^{s/2-3/2+\theta}\|\widetilde J\|_{L^2(D_{1/2})},
\end{split}
\end{equation*}
which proves \eqref{eq:bc control 1} with
\[
\zeta_1=2(2+\frac{s}{2}-\frac{3}{2}+\theta)=\min\bigg(s+1, s+5-2(t+r), 5-2t-s\bigg),
\]
\[
\zeta_2=2+\min\bigg( \frac{2s-1}{2}-\frac{3}{2}+\theta, -t \bigg)=\min\bigg( s, s+2-t-r, 2-t \bigg ).
\]
\end{proof}

Next, we prove the Sobolev extension that was needed in the proof of Lemma~\ref{lem:bc control}.
\begin{lem}\label{lem:extension}
For any $\phi\in H^{1/2}(\partial D)^3$, there exists $U\in H^2(\Omega)^3$ such that
\begin{enumerate}[(i)]
\item $\nu\wedge U=0$\quad on\ \ $\partial D$,

\item $\nu\wedge(\nu\wedge (\nabla\wedge U))=\nu\wedge(\nu\wedge\phi)$\quad on $\partial D$,

\item $\|U\|_{H^2(D)^3}\leq C\|\phi\|_{H^{1/2}(\partial D)^3}$. 
\end{enumerate}
In (iii), $C$ is a positive constant depending only on $D$. 
\end{lem}
\begin{proof}
First, we let $(V,p)\in H^1(D)^3\wedge L^2(D)$ be the solution to the following Stokes system (cf. \cite{BVH})
\begin{equation}\label{eq:stokes}
\begin{cases}
-\Delta V+\nabla p=0\quad &\mbox{in\ \ $D$},\\
\text{div}\ V=0\quad &\mbox{in\ \ $D$},\\
V=\nu\wedge(\nu\wedge\phi)\quad &\mbox{on\ \ $\partial D$}.
\end{cases}
\end{equation}
Moreover, we have
\begin{equation}\label{eq:well1}
\|V\|_{H^1(D)^3}\leq C\|\phi\|_{H^{1/2}(\partial D)^3},
\end{equation}
where $C$ is a positive constant depending only on $D$. Next, we introduce the following auxiliary system
\begin{equation}\label{eq:auxiliary2}
\begin{cases}
\nabla\wedge(\nabla\wedge U)=\nabla\wedge V\quad &\mbox{in\ \ $D$},\\
\text{div}\ U=0\quad &\mbox{in\ \ $D$},\\
\nu\cdot U=0\quad &\mbox{on\ \ $\partial D$},\\
\nu\wedge U=0\quad&\mbox{on\ \ $\partial D$}.
\end{cases}
\end{equation}
Referring to Section 1.5 in \cite{CT}, we know there exists a solution $U\in H^2(D)^3$ to the system \eqref{eq:auxiliary2} and
\begin{equation}\label{eq:well2}
\|U\|_{H^2(D)^3}\leq C\|\nabla\wedge V\|_{L^2(\Omega)^3},
\end{equation}
where $C$ is a constant depending only on $D$.
We shall show $\nabla\wedge U=V$. To that end, we first note that (cf. \cite{BCS} and \cite{ColKre}) 
\begin{equation}\label{eq:well3}
\nu\cdot(\nabla\wedge U)=-\text{Div}(\nu\wedge U)=0\quad\mbox{on\ \ $\partial D$}. 
\end{equation}
Since 
\[
\nabla\wedge(\nabla\wedge U-V)=0\quad\mbox{in\ \ $D$},
\]
by Theorem~3.37 in \cite{Mon}, we know there exists $u\in H^1(D)$ such that
\begin{equation}\label{eq:well4}
\nabla\wedge U-V=\nabla u\quad\mbox{in \ \ $D$}. 
\end{equation}
Clearly, we also have $u\in H^2(D)$. By taking the divergence of both sides of \eqref{eq:well4}, one has
\begin{equation}\label{eq:well5}
\Delta u=0\quad\mbox{in\ \ $D$}.
\end{equation}
On the other hand, by taking inner-product of both sides of \eqref{eq:well4} with $\nu$, one further has
\[
\frac{\partial u}{\partial \nu}=\nu\cdot(\nabla\wedge U)-\nu\cdot V=0\quad\mbox{on\ \ $\partial D$}, 
\]
which together with \eqref{eq:well5} immediately implies $\nabla u=0$ in $D$. Therefore, one has from \eqref{eq:well4} that
\begin{equation}\label{eq:well6}
\nabla\wedge U=V\quad\mbox{in\ \ $D$}.
\end{equation}

By \eqref{eq:stokes} and \eqref{eq:well6}, we obviously have $\nu\wedge(\nu\wedge(\nabla\wedge U))=\nu\wedge(\nu\wedge V)=\nu\wedge(\nu\wedge\phi)$ on $\partial D$, which together with \eqref{eq:auxiliary2} and \eqref{eq:well2} readily indicates that $U$ fulfill all the requirements of the extension function stated in the lemma.

The proof is complete. 

\end{proof}

\begin{lem}\label{lem:small inclusion}
For $\tau\in\mathbb{R}_+$, 
let $E_\tau, H_\tau \in H_{loc}(\nabla\wedge; \mathbb{R}^3\backslash\overline{D}_\tau)$
be the solutions to the following scattering problem
\begin{equation}\label{eq:small inclusion 1}
\begin{cases}
\displaystyle{\nabla\wedge E_\tau^+-i\omega\mu_0 H_\tau^+=0}\qquad\qquad &\mbox{in\ \ $\mathbb{R}^3\backslash\overline{D}_\tau$,}\\
\displaystyle{\nabla\wedge H_\tau^++i\omega \varepsilon_0 E_\tau^+=0}\quad &\mbox{in\ \ $\mathbb{R}^3\backslash\overline{D}_\tau$},\\
\nu\wedge E_\tau^+=\psi\in TH^{-1/2}_{\text{\emph{Div}}}(\partial D_\tau)\qquad & \mbox{on\ \ $\partial D_\tau$},\\
\displaystyle{\lim_{|x|\rightarrow\infty}|x|\Big| (\nabla\wedge E_\tau^+)(x)\wedge\frac{x}{|x|}-i\omega E_\tau^+(x) \Big|=0.}
\end{cases}
 \end{equation}
 Then there exists $\tau_0\in\mathbb{R}_+$ such that the following estimate holds for $\tau<\tau_0$,
 \begin{equation}\label{eq:small1}
 \|\nu\wedge E_\tau\|_{TH^{-1/2}_{\text{\emph{Div}}}(\partial B_R)}\leq C \tau^{2}\|\psi(\tau\ \cdot)\|_{H^{-1/2}(\partial D)^3}\,.
 \end{equation}
Moreover, if $\psi(x)=E^i(x)$ is the solution to \eqref{eq:Maxwell incident} it holds that 
\begin{equation}\label{eq:small2}
 \|\nu\wedge E_\tau\|_{TH^{-1/2}_{\text{\emph{Div}}}(\partial B_R)}\leq C \tau^{3}\|\nu\wedge E^i\|_{TH^{-1/2}_{\text{\emph{Div}}}(\partial B_R)}\,.
 \end{equation}
The constants $C$ in (\ref{eq:small1})-(\ref{eq:small2}) are generic, depending only on $R, \omega$, $\tau_0$ and $D$. 
\end{lem}
\begin{proof}
The proof follows a natural modification of the estimates derived in \cite[Section 3]{BL}.
\end{proof}

\begin{rem}
For the results in Lemma~\ref{lem:small inclusion}, 
we would  like to mention some closely related studies on the scattering estimates 
due to small EM scatterers in \cite{Ammari4,AmmVogVol}, and on the low-frequency asymptotics of 
EM scattering in \cite{AmmNed,Kle,Mar,Ned}). 
\end{rem}

We are now ready to prove the main result of this work, Theorem~\ref{thm:main}.
For the sake of exposition, we refer to the system \eqref{eq:Maxwell virtual} 
as the scattering problem in the virtual space and denote by $A^\rho_\infty(\hat{x}):=A_\infty(\hat{x}; E^i, (\Omega; \varepsilon_\rho,$ $\mu_\rho, \sigma_\rho), J)$ 
the corresponding scattering amplitude. 
Noting that mapping $F$ (see (\ref{eq:map whole})) is identity outside $\Omega$, 
we know 
$(E_\rho, H_\rho)=(\widetilde E_\rho, \widetilde H_\rho)$ in $\mathbb{R}^3\backslash\overline{\Omega}$, 
and hence
\begin{equation}\label{eq:equivalence vp}
A^\rho_\infty(\hat{x}; E^i)=\widetilde A^\rho_\infty(\hat{x}; E^i), \quad \hat{x}\in\mathbb{S}^2.
\end{equation}
Using these relations, it is easy to see that Theorem~\ref{thm:main} is a direct consequence of
the following theorem.

\begin{thm}\label{thm:main virtual}
Let $(\mathbb{R}^3; \varepsilon, \mu, \sigma)$ be the EM medium described in
\eqref{eq:virtual medium}--\eqref{eq:virtual tensors}, and $J$ be the current density given in \eqref{eq:current virtual}, satisfying \eqref{eq:current assumption virtual}, and $A^\rho_\infty(\hat{x})$ be
the scattering amplitude corresponding to $E_\rho^+$ in \eqref{eq:Maxwell virtual}. 
Then there exists a positive constant $\rho_0$ such that the following estimate holds 
for $\rho<\rho_0$,
\begin{equation}\label{eq:main estimate virtual}
|A^\rho_\infty(\hat{x}; E^i)|
\leq C\bigg(  \rho^{\min(\zeta_1, 3)} \|E^i\|_{H(\nabla\wedge;\Omega)}
+ \rho^{\frac{\zeta_1}{2}}  \|\widetilde{J}\|_{L^2(D_{1/2})^3}+\rho^{\zeta_2}\|\widetilde J \|_{L^2(D\backslash D_{1/2})^3}  \bigg )
\end{equation}
where $\zeta_1$ and $\zeta_2$ are given in \eqref{eq:zeta1}--\eqref{eq:zeta2}, and $C$ is a positive constant depending only on $\alpha, \beta, \gamma, \omega$, $c_0$ in \eqref{eq:current assumption virtual}, $C_0$ in \eqref{eq:uniform regular condition} and $\Omega, D$, 
but independent of $\rho, r, s, t$ and $\varepsilon_a, \mu_a, \sigma_a$, $\widetilde J$, $E^i$.
\end{thm}

\begin{proof}
Let $E_1^+$, $H_1^+\in H_{loc}(\nabla\wedge; \mathbb{R}^3\backslash\overline{D}_\rho)$ 
and $E_2^+$, $H_2^+\in H_{loc}(\nabla\wedge; \mathbb{R}^3\backslash\overline{D}_\rho)$ 
be the solutions to the following two Maxwell scattering systems respectively,
\begin{equation}\label{eq:f1}
\begin{cases}
\displaystyle{\nabla\wedge E_1^+-i\omega\mu_0 H_1^+=0}\qquad\qquad &\mbox{in\ \ $\mathbb{R}^3\backslash\overline{D}_\rho$,}\\
\displaystyle{\nabla\wedge H_1^++i\omega \varepsilon_0 E_1^+=0}\quad &\mbox{in\ \ $\mathbb{R}^3\backslash\overline{D}_\rho$},\\
\nu\wedge E_1^+=\nu\wedge E_\rho \in TH^{-1/2}_{\text{Div}}(\partial D_\rho)\qquad & \mbox{on\ \ $\partial D_\rho$},\\
\displaystyle{\lim_{|x|\rightarrow+\infty}|x|\left| (\nabla\wedge E_1^+)(x)\wedge\frac{x}{|x|}-i\omega E_1^+(x) \right|=0}\,, 
\end{cases}
 \end{equation}
 \begin{equation}\label{eq:f1}
\begin{cases}
\displaystyle{\nabla\wedge E_2^+-i\omega\mu_0 H_2^+=0}\qquad\qquad &\mbox{in\ \ $\mathbb{R}^3\backslash\overline{D}_\rho$,}\\
\displaystyle{\nabla\wedge H_2^++i\omega \varepsilon_0 E_2^+=0}\quad &\mbox{in\ \ $\mathbb{R}^3\backslash\overline{D}_\rho$},\\
\nu\wedge E_2^+=\nu\wedge E^i \in TH^{-1/2}_{\text{Div}}(\partial D_\rho)\qquad & \mbox{on\ \ $\partial D_\rho$},\\
\displaystyle{\lim_{|x|\rightarrow+\infty}|x|\left| (\nabla\wedge E_2^+)(x)\wedge\frac{x}{|x|}-i\omega E_2^+(x) \right|=0.}
\end{cases}
 \end{equation}
 It is easy to see that 
 \begin{equation}\label{eq:f2}
 E_\rho^+=E_1^+-E_2^+\qquad\mbox{in} \quad \mathbb{R}^3\backslash\overline{D}_\rho.
 \end{equation}
By taking $\tau=\rho$ in Lemma~\ref{lem:small inclusion} and using Lemma~\ref{lem:bc control}, we have
\begin{equation}\label{eq:g1}
\begin{split}
&\| \nu\wedge E_\rho^+ \|_{TH^{-1/2}(\partial B_R)}\\
\leq & C_1  \rho^{\frac{\zeta_1}{2}} \bigg\{ \|\nu\wedge E_\rho^+\|^{1/2}_{TH^{-1/2}_{\text{{Div}}}(\partial B_R)} \|\Lambda(\nu\wedge E_\rho^+)\|^{1/2}_{TH^{-1/2}_{\text{{Div}}}(\partial B_R)}\\
&\qquad\ \ + \|\nu\wedge E^i\|^{1/2}_{TH^{-1/2}_{\text{{Div}}}(\partial B_R)}\|\Lambda(\nu\wedge E_\rho^+)\|^{1/2}_{TH^{-1/2}_{\text{{Div}}}(\partial B_R)}\\
&\qquad\ \ + \|\nu\wedge H^i\|^{1/2}_{TH^{-1/2}_{\text{{Div}}}(\partial B_R)}\|\nu\wedge E_\rho^+\|^{1/2}_{TH^{-1/2}_{\text{{Div}}}(\partial B_R)}\bigg\}\\
&+C_1\rho^3\|\nu\wedge E^i\|_{TH^{-1/2}_{\text{Div}}(\partial B_R)}+C_1\rho^{\frac{\zeta_1}{2}}\|\widetilde J\|_{L^2(D_{1/2})^3}+C_1\rho^{\zeta_2}\|\widetilde J \|_{L^2(D\backslash D_{1/2})^3}.
\end{split}
\end{equation}
In the sequel, we let
\begin{equation}\label{eq:g2}
\|\Lambda\|_{\mathcal{L}(TH^{-1/2}(\partial B_R), TH^{-1/2}(\partial B_R) )}\leq \epsilon_0.
\end{equation}
Then it follows from \eqref{eq:g1} and \eqref{eq:g2} that 
\begin{equation}\label{eq:g4}
\begin{split}
& \|\nu\wedge E_\rho^+\|_{TH^{-1/2}(\partial B_R)}\\
\leq & C_1\epsilon_0\rho^{\frac{\zeta_1}{2}}  \|\nu\wedge E_\rho^+\|_{TH^{-1/2}(\partial B_R)}+ C_1^2\epsilon_0 \rho^{\zeta_1}\|\nu\wedge E^i\|_{TH^{-1/2}_{\text{Div}}(\partial B_R)}\\
&+\frac 1 4  \|\nu\wedge E_\rho^+\|_{TH^{-1/2}(\partial B_R)}+C_1^2\rho^{\zeta_1}  \|\nu\wedge H^i\|_{TH^{-1/2}(\partial B_R)}\\
&+\frac 1 4  \|\nu\wedge E_\rho^+\|_{TH^{-1/2}(\partial B_R)}+C_1\rho^3\|\nu\wedge E^i\|_{TH^{-1/2}_{\text{Div}}(\partial B_R)}\\
&+C_1\rho^{\frac{\zeta_1}{2}}\|\widetilde J\|_{L^2(D_{1/2})^3}+C_1\rho^{\zeta_2}\|\widetilde J \|_{L^2(D\backslash D_{1/2})^3}.
\end{split}
\end{equation}
By taking $\rho_0\in\mathbb{R}_+$ to be sufficiently small such that 
$C_1\epsilon_0 \rho^{\zeta_1/2}<1/4$, then the first, third and fifth terms in the RHS of estimate 
\eqref{eq:g4} can be absorbed by the LHS, leading to 
the existence of a constant $C_2>0$ such that
\begin{eqnarray}
&&\|\nu\wedge E_\rho^+\|_{TH^{-1/2}(\partial B_R)}\nb\\
&\leq & C_2\rho^{\zeta_1}\left( \|\nu\wedge E^i\|_{TH^{-1/2}(\partial B_R)}+\|\nu\wedge H^i\|_{TH^{-1/2}(\partial B_R)} \right)
+ C_2 \rho^3\|\nu\wedge E^i\|_{TH_{\text{Div}}^{-1/2}(\partial B_R)}\nb\\
&&+C_1\rho^{\frac{\zeta_1}{2}}\|\widetilde J\|_{L^2(D_{1/2})^3}+C_1\rho^{\zeta_2}\|\widetilde J \|_{L^2(D\backslash D_{1/2})^3}. \label{eq:g5}
\end{eqnarray}
We can directly verify that $E^i$ and $H^i$ satisfies the vector-valued Helmholtz equtions
\[
\Delta E^i+\omega^2 E^i=0, \quad \Delta H^i+\omega^2 H^i=0\quad\mbox{in\ $\Omega$}\,, 
\]
then have the estimate by the interior estimates for elliptic equations that
\begin{equation}\label{eq:g6}
\begin{split}
&\|\nu\wedge E^i\|_{TH^{-1/2}_{\text{Div}}(\partial B_R)}+\|\nu\wedge H^i\|_{TH^{-1/2}_{\text{Div}}(\partial B_R)}\\
\leq & C_3\left( \|E^i\|_{H(\nabla\wedge; B_R)}+\|H^i\|_{H(\nabla\wedge; B_R)} \right)\\
\leq & C_4\left( \|E^i\|_{L^2(\Omega)}+\|H^i\|_{L^2(\Omega)} \right )\\
\leq & C_5 \|E^i\|_{H(\nabla\wedge;\Omega)},
\end{split}
\end{equation}
where $C_3, C_4$ and $C_5$ are generic positive constants depending only on $R$, $\Omega$ and $\omega$. By combining \eqref{eq:g5} and \eqref{eq:g6}, one readily has that
\begin{equation}\label{eq:g7}
\begin{split}
\|\nu\wedge E_\rho^+\|_{TH^{-1/2}(\partial B_R)}\leq C_6\bigg( & \rho^{\min(\zeta_1, 3)} \|E^i\|_{H(\nabla\wedge;\Omega)}\\
& + \rho^{\frac{\zeta_1}{2}}  \|\widetilde{J}\|_{L^2(D_{1/2})^3}+\rho^{\zeta_2}\|\widetilde J \|_{L^2(D\backslash D_{1/2})^3}  \bigg ).
\end{split}
\end{equation}
Moreover, we know by \eqref{eq:g2} that 
\begin{equation}\label{eq:g8}
\begin{split}
\|\nu\wedge H_\rho^+\|_{TH^{-1/2}(\partial B_R)}\leq C_6\epsilon_0\bigg( & \rho^{\min(\zeta_1, 3)} \|E^i\|_{H(\nabla\wedge;\Omega)}\\
& + \rho^{\frac{\zeta_1}{2}}  \|\widetilde{J}\|_{L^2(D_{1/2})^3}+\rho^{\zeta_2}\|\widetilde J \|_{L^2(D\backslash D_{1/2})^3}  \bigg ).
\end{split}
\end{equation}
Now the desired estimate \eqref{eq:main estimate virtual} follows 
directly from \eqref{eq:g7}--\eqref{eq:g8} and the following integral representation (cf.\,\cite{ColKre})
\begin{equation}
A_\infty^\rho(\hat{x})=\frac{i\omega}{4\pi}\hat{x}\wedge\int_{\partial B_R}\bigg\{\nu(y)\wedge E_\rho^+(y)+(\nu(y)\wedge H_\rho^+(y)\wedge\hat{x}\bigg\} e^{-i\omega\hat{x}\cdot y}\ ds_y\,.
\end{equation}
\end{proof}

\section*{Acknowledgements}
The research of GB was supported
in part by the NSF grants DMS-0908325, DMS-0968360,
DMS-1211292,  the ONR grant N00014-12-1-0319, a Key
Project of the Major Research Plan of NSFC (No. 91130004), and a
special research grant from Zhejiang University. The research of HL is supported by NSF grant DMS-1207784. The work of JZ was supported by Hong Kong RGC grant
(Project 405110) and the CUHK Focused Investment Scheme 2012/2014.

\end{document}